\documentclass[10pt]{article}
\usepackage{ifthen}
\usepackage{commath}
\usepackage{tikz}
\usepackage{amsmath}
\usepackage{yfonts}
\usepackage{mathrsfs}
\usepackage{verbatim}
\usepackage{amsthm}
\usepackage{amssymb}
\usepackage{graphicx}
\DeclareMathOperator{\alinv}{\alpha^{-1}}

\newtheorem{thm}{Theorem}
\newtheorem{rem}{Remark}

\newtheorem{lem}{Lemma}

\newtheorem{defn}{Definition}

\begin{document}

\title{Universal imbedding of a Hom-Lie Triple System}

\author{R. Vandermolen}

\date{}

\maketitle

\begin{abstract}
In this article we will build a universal imbedding of a regular Hom-Lie triple system into a Lie algebra and show that the category of regular Hom-Lie triple systems is equivalent to a full subcategory of pairs of $\mathbb{Z}_2$-graded Lie algebras and Lie algebra automorphism, then finally give some characterizations of this subcategory. 
\end{abstract}

\section*{Introduction}

\noindent
Ternary algebras, that is algebras with ternary multiplication are studied heavily in Lie and Jordan theories, geometry, analysis and physics. For example, Jordan triple system (\cite{jacob2},\cite{wg},\cite{quant1},\cite{extratrip}) can be realized as 3-graded Lie algebras through the TKK construction (\cite{oleg2}), from which special Lie algebras can be obtained. Also more to the heart of this study are Lie triple systems, which give rise to $\mathbb{Z}_2$-graded Lie algebras (see \cite{jacob}, \cite{oleg}), which are Lie algebras associated to symmetric spaces. 

Hom-Lie algebras and similar structures have been more recently studied, (\cite{tau},\cite{sigma}), as they lead to many interesting deformations of Lie algebras, in particular $q$-deformations which have been studied by both mathematicians and physicists (\cite{quant2},\cite{amm},\cite{quant4}). The $q$-deformed Lie algebras have been of interest in supersymmetirc gauge theory \cite{quant2}, the Yang-Baxter equation (\cite{tau2},\cite{nambu}) and especially with $q$-defermations of the Witt and Virasoro algebras (\cite{quant3},\cite{quant5},\cite{quant6},\cite{quant7}).

Extending the theory of deformations into the realm of ternary algebras has been done with the use of ternary Hom-Nambu Lie algebras in papers like (\cite{joa},\cite{amm}). Most recently it has been attempted by Yau in \cite{tau} with Hom-Lie triple systems (in the form we present them here). Here we will expand this theory by constructing a universal imbedding of a regular Hom-Lie triple system into a $\mathbb{Z}_2$ graded Lie algebra. 

While at first inspection the regular assumption may seem to be inordinately restrictive, especially since it is most popular in the research of this field to only use the assumption of multiplicative.  Yet, it is immediately present that in the finite dimensional case any simple multiplicative Hom-Lie triple systems (or regular Hom-Lie algebra) is regular. In addition, any finite dimensional multiplicative Hom-Lie triple system (reps. Hom-Lie algebra) is an extension of a regular Hom-Lie triple system (regular Hom-Lie algebra) by a Hom-Lie triple system (Hom-Lie algebra) with the trivial twisting map.

Next we will discuss the organization of this paper. In the first section basic definitions of Hom-Lie algebras, and Hom-Lie triple systems and a few consequences are considered. Furthermore, we review homotopes and isotopes of algebras, which we find as a natural setting for Hom-objects. In the second section we recall some results from \cite{tau} which show some methods of constructing Hom-Lie triple systems from Lie algebras, Hom-Lie algebras from Lie algebras, Lie algebras from Hom-Lie algebras, and even constructing Lie algebras from Hom-Lie triple systems. In the third section we introduce the concept of universal imbedding and build an imbedding for a Hom-Lie triple system into a Lie algebra, and a category isomorphism between regular Hom-Lie algebras and Lie algebras. While in the fourth section we construct the universal imbedding of a Hom-Lie triple system and finally in section five we describe the category of Lie algebras that is equivalent to the Hom-Lie triple systems. 

\section{Basic Definitions}
In this section we recall a few definitions. Throughout this article $k$ will  be a commutative ring with a multiplicative identity, have $\text{Char}(k)\neq 2$, and will commonly be omitted. 

Our fist definition is that of a Hom-Lie algebra, which we will see is equivalent to Lie algebra under a certain condition. Yet for our treatment is all we need.

\subsection{Hom-Lie Algebras}
\begin{defn}
\emph{A} Hom-Lie Algebra \emph{is the triple $($$L$,$[$ , $], \alpha)$ consisting of a $k$-module $L$, a bilinear map $[$ , $]:L\times L\rightarrow L$, and a linear map $\alpha:L\rightarrow L$ such that for all $a,b,c\in L$.}

\begin{equation}\label{lie1}
 [a,a]=0
 \end{equation}
 \begin{equation}\label{lie2}
[\alpha(a),[b,c]]+[\alpha(b),[c,a]]+[\alpha(c),[a,b]]=0
\end{equation}

\end{defn}
\begin{rem}
\emph{Identity (\ref{lie1}) with our conditions on $k$, i.e. $\text{Char}(k)\neq 2$, is equivalent to}
\begin{equation*}
[a,b]=-[b,a]
\end{equation*}
\emph{for all $a,b\in L$, and is known as the} skew symmetry identity.  \emph{While identity (\ref{lie2}) is known as the the} twisted Jacobi identity. 
\end{rem}
\begin{rem}
\emph{Notice that in the above definition when $\alpha=Id_L$ this definition is just the definition of a Lie algebra.}
\end{rem}
\begin{defn}
\emph{Let $(L,[$ , $]_L, \alpha_L)$ and $(L',[$ , $]_{L'}, \alpha_{L'})$ be two Hom-Lie algebras. A linear map $f:L\rightarrow L'$ is a} Hom-Lie algebra homomorphism \emph{whenever for all $a,b\in L$}
\begin{equation*}
[f(a),f(b)]_{L'}=f([a,b]_L).
\end{equation*}
\emph{and}
\begin{equation*}
f\circ\alpha=\alpha\circ f
\end{equation*}
\emph{that is the following two diagrams}

\vspace{5mm}
\begin{center}
\begin{tikzpicture}[scale=2]
\draw(0,0) node[right] {$L\otimes L$};
\draw[->] (.7,0)--(1.6,0);
\draw(1.6,0)node[right]{$L'\otimes L'$};
\draw[->](.33,-.15)--(.33,-.8);
\draw(.33,-.8)node[below]{$L$};
\draw[->] (.46,-.95)--(1.8,-.95);
\draw[->](1.98,-.16)--(1.98,-.8);
\draw(2,-.5)node[right]{$[$ , $]_{L'}$};
\draw(1,-1)node[below]{$f$};
\draw(.31,-.48)node[left]{$[$ , $]_L$};
\draw(1.8,-.95)node[right]{$L'$};
\draw(.78,.19)node[right]{$f\otimes f$};

\draw(4,0) node[right] {$L$};
\draw[->] (4.3,0)--(5,0);
\draw(5,0)node[right]{$L'$};
\draw[->](5.2,-.14)--(5.2,-.8);
\draw(4.54,.15)node[right]{$f$};
\draw(5.2,-.44)node[right]{$\alpha_{L'}$};
\draw[->](4.14,-.14)--(4.14,-.8);
\draw(4.14,-.44)node[left]{$\alpha_L$};
\draw(4,-.98)node[right]{$L$};
\draw[->](4.3,-.98)--(5,-.98);
\draw(4.54,-1.15)node[right]{$f$};
\draw(5.04,-.98)node[right]{$L'$};

\end{tikzpicture}
\end{center}
\emph{both commute.}
\end{defn}

\begin{defn}
\emph{ A Hom-Lie algebra, $(L$, [ , ], $\alpha$), is called}
\begin{itemize}
  \item multiplicative \emph{when $\alpha$ is a Hom-Lie algebra homomorphism.}
  \item regular \emph{when $\alpha$ is a Hom-Lie algebra automorphism.}
\end{itemize}
\end{defn}

\subsection{Hom-Lie Triple Systems}
Our next definition was first introduced by Yau in \cite{tau}.

\begin{defn}
\emph{A} Hom-Lie Triple System \emph{is the triple $($$T$,$[$ , , $]$,$\alpha=(\alpha_1, \alpha_2$$))$ consisting of an $k$-module $T$, a trilinear map $[$ , , $]:T\times T\times T\rightarrow T$, and a pair of linear maps $\alpha_i:T\rightarrow T$ for $i=1,2$, called the twisting maps, such that for all $a,b,c,d,e\in T$,}

\begin{equation}\label{lts1}
[a,a,b]=0
\end{equation}
\begin{equation}\label{lts2}
[a,b,c]+[b,c,a]+[c,a,b]=0
\end{equation}
\label{lts3}
\begin{align}\label{lts3}
[\alpha_1(a),\alpha_2(b),[c,d,e]]=[[a,b,c],\alpha_1(d),\alpha_2(e)] & +[\alpha_1(c),[a,b,d],\alpha_2(e)]\\
& +[\alpha_1(c),\alpha_2(d),[a,b,e]]
\end{align}

\end{defn}

\begin{rem}
\emph{Notice again under our assumptions identity (\ref{lts1}) is equivalent to}
\begin{equation*}
[a,b,c]=-[b,a,c]
\end{equation*}
\emph{for all $a,b,c\in T$, and is known as the} left skew symmetry identity. \emph{while (\ref{lts3}) is known as the} ternary Hom-Nambu identity.
\end{rem}
\begin{rem}
\emph{For brevity we will refer to Hom-Lie Triple System as simply a}  Hom-LTS.
\end{rem}

\begin{defn}
\emph{Let  $(T, [$ , , $]_{T},\alpha)$ and $(T',[$ , , $]_{T'},\alpha')$ be two Hom-LTS. A linear map $f:T\rightarrow T'$ is a} Hom-LTS homomorphism  \emph{ whenever for all $a,b,c\in T$}
\begin{equation*}
[f(a),f(b),f(c)]_{T'}=f([a,b,c]_{T})  
\end{equation*}
\emph{and for $i\in\{1,2\}$, we have}
\begin{equation*}
 f\circ\alpha_i =\alpha'_i\circ f
 \end{equation*}
 \emph{that is the following two diagrams,}

\vspace{5mm}
\begin{center}
\begin{tikzpicture}[scale=2]
\draw(0,0) node[right] {$T\otimes T\otimes T$};
\draw[->] (1.1,0)--(2.33,0);
\draw(2.33,0)node[right]{$T'\otimes T'\otimes T'$};
\draw[->](.52,-.15)--(.52,-.8);
\draw(.52,-.8)node[below]{$T$};
\draw[->] (.6,-.95)--(2.78,-.95);
\draw[->](2.9,-.16)--(2.9,-.8);
\draw(2.9,-.5)node[right]{$[$ , , $]_{T'}$};
\draw(1.7,-1)node[below]{$f$};
\draw(.5,-.48)node[left]{$[$ , , $]_T$};
\draw(2.75,-.95)node[right]{$T'$};
\draw(1.2,.19)node[right]{$f\otimes f\otimes f$};

\draw(5,0) node[right] {$T$};
\draw[->] (5.3,0)--(6,0);
\draw(6,0)node[right]{$T'$};
\draw[->](6.2,-.14)--(6.2,-.8);
\draw(5.54,.15)node[right]{$f$};
\draw(6.2,-.44)node[right]{$\alpha'_i$};
\draw[->](5.14,-.14)--(5.14,-.8);
\draw(5.14,-.44)node[left]{$\alpha_i$};
\draw(5,-.98)node[right]{$T$};
\draw[->](5.3,-.98)--(6,-.98);
\draw(5.54,-1.15)node[right]{$f$};
\draw(6.04,-.98)node[right]{$T'$};

\end{tikzpicture}
\end{center}
\emph{both commute.}
\end{defn}

\begin{defn}
\emph{Adopting terminology from Hom-Lie algebras, considering $(T$, [ , , ], $\alpha$) a Hom-LTS, where $\alpha_1=\alpha_2=\alpha$, the Hom-LTS is called}
\begin{itemize}
  \item multiplicative \emph{when $\alpha$ is a Hom-LTS homomorphism.}
  \item regular \emph{when $\alpha$ is a Hom-LTS automorphism.}
\end{itemize}
\end{defn}

\begin{rem}
\emph{For $($$T$, $[$ , , $]$,$\alpha$$)$ a regular Hom-Lie triple system  and $a,b,c,d,e\in T$ we have the following:}
\begin{equation}\label{truths}
[a,b[c,d,e]^{\alinv}]-[c,d,[a,b,e]^{\alinv}]=[c,[a,b,d]^{\alinv}, d]+[[a,b,c]^{\alinv},d,e]
\end{equation}
\emph{where $[a,b,c]^{\alinv}\doteq\alinv([a,b,c])=[\alinv(a),\alinv(b),\alinv(c)]$.}
\end{rem}

\subsection{Homotope Theory}
Here we review some definitions of homotope theory that will be useful to us. 
\begin{defn}
\emph{Given an associative unital $k$-algebra, $\mathcal{A}$, and an element $a\in\mathcal{A}$, the algebra $\mathcal{A}^a$ will be called} the homotope of $\mathcal{A}$. \emph{The homotope retains the same module structure of $\mathcal{A}$, yet we replace the composition of $\mathcal{A}$ by}
\begin{equation*}
x\cdot_a y\doteq x\circ a\circ y
\end{equation*}
\emph{for $x,y\in\mathcal{A}$. When $a$ is invertible we call a homotope an} isotope.
\end{defn}
\begin{rem}
\emph{It is immediately clear that from the associtivity of $\mathcal{A}$, that $\mathcal{A}^a$ is also associative.} 
\end{rem}
\begin{defn}
\emph{From the preceding remark we may then form the} homotope Lie algebra \emph{with the standard commutator, that is  $(\mathcal{A}^a, [$ , ]) is a Lie algebra with the commutator defined as,}
\begin{equation*}
[x,y]=x\cdot_a y-y\cdot_a x=x\circ a\circ y-y\circ a\circ x
\end{equation*}
\emph{for $x,y\in\mathcal{A}$}.
\end{defn}

\section{Inducing Lie algebras and Hom-LTS}
There are many constructions using Hom-Lie algebras for making Hom-LTS and ternary Hom-Nambu algebras, in papers like \cite{tau} and \cite{joa}. Yet they both rely on the the trivial case for the twisting map to draw the connection between Lie algebras and Hom-LTS.  We will use this section to show the strong connection between Hom-LTS and strictly Lie algebras. 

\subsection{Hom-Lie algebra induced by a Lie algebra}
The following is Lemma 2.8 from \cite{tau}, in the case that is of most interest to us.
\begin{lem}\label{ho}
\emph{Let $(L,[$ , $])$ be a Lie algebra and $\alpha:L\rightarrow L$ be a Lie algebra homomorphism. Then ($L, [$ , $]_\alpha=\alpha\circ [$ , $],\alpha)$ is a multiplicative Hom-Lie algebra. Furthermore if $\alpha$ is a Lie algebra automorphism then ($L, [$ , $]_\alpha=\alpha\circ [$ , $],\alpha)$ is a regular Hom-Lie algebra.}
\end{lem}

\subsection{Lie algebra induced by a Hom-Lie algebra}
The next Lemma is a rewording of Theorem 2.5 in \cite{tau}, for the case we will be first examining. 
\begin{lem}\label{hoho}
\emph{Let ($L, [$ , $],\alpha)$ be a regular Hom-Lie algebra then ($L$, [ , $]_\alpha$) is a Lie algebra, where $[a,b]_\alpha=\alinv([a,b])$ for all $a,b\in L$. }
\end{lem}

\subsection{Hom-LTS induced by a Lie algebra}
In \cite{tau} Corollary 3.14 Yau uses a standard construction of a Jordan Algebra to make a multiplicative Hom-Jordan triple system. The following Lemma is an analogous result for Lie algebras and Hom-LTS.
\begin{lem}\label{funclem}
\emph{Let $(L,[$ , $])$ be a Lie algebra and $\alpha : A\rightarrow A$ be a Lie algebra homomorphism. Then $(L,[$ , , $]_\alpha=\alpha\circ[$ , , $], \alpha )$ is a multiplicative Hom-LTS, where for $a,b,c\in L$ we define $[a,b,c]=[[ab],c]$. Furthermore if $\alpha$ is a Lie automorphism then $(L,[$ , , $]_\alpha=\alpha\circ[$ , , $], \alpha )$ is a regular Hom-LTS. }
\end{lem}

\subsection{Lie Algebra induced by a Hom-LTS}
Before diving into this next construction we will make the assumption for the rest of the paper that $\alpha$ is an algebra automorphism. 

\begin{defn}
\emph{For a regular Hom-Lie triple system $($$T$, $[$ , , $]$,$\alpha$$)$  we will define $\mathcal{HDR}(T)\doteq\{X\in End(T) |\; (X\circ\alinv)[cde]=[X^{\alinv}c,d,e]+[c,X^{\alinv}d,e] +[c,d,X^{\alinv} e] \}$ where $X^{\alinv}\doteq\alinv\circ X$. It is immediately obvious that this is a closed sub-vector space of $End(T)$.}
\end{defn}

\begin{rem}
\emph{It turns out that $\mathcal{HDR}(T)$ is a sub-lie algebra of the isotope Lie algebra of $\text{End}(T)^{\alinv}$. That is for $X,Y\in\mathcal{HDR}(T)$ we have $[X,Y]\in\mathcal{HDR}(t)$ where $[X,Y]\doteq X\circ\alinv\circ Y-Y\circ\alinv\circ X$.}
\end{rem}

\begin{defn}
\emph{Furthermore we will define $\mathcal{IHD}(T)\doteq\{ D_{ab} |\; a,b\in T\}$, where $D_{ab}(c)=[a,b,c]$.}
\end{defn}

\begin{rem}
\emph{Notice that $\mathcal{IHD}(T)$ is an ideal of $\mathcal{HDR}(T)$ since for $X\in\mathcal{HDR}(T)$ and $D_{ab}\in\mathcal{IHD}(T)$ we have the following.}
\begin{equation}\label{need1}
[X,D_{ab}]=D_{aX^{\alinv}(b)}+D_{X^{\alinv}(a),b}
\end{equation}
\emph{Which makes $\mathcal{IHD}(T)$ a Lie algebra.}
\end{rem}

\begin{lem}\label{pooper}
\emph{For a regular Hom-LTS $($$T$, $[$ , , $]$,$\alpha$$)$  we can induce a Lie algebra $\mathcal{IHD}(T)$ defined as above.}
\end{lem}

\section{The Category of Lie algebras, Hom-LTS and Hom-Lie algebras}
For our next construction we will need to define the category of Lie algebras we will be concerned with. 
\begin{defn}
\emph{We will denote, $\mathcal{TL}$ as the category with the pairs $(L, \alpha)$, where $L$ is a Lie algebra and $\alpha:L\rightarrow L$ is a Lie algebra automorphism, as objects and a Lie algebra homomorphism $f:L_1\rightarrow L_2$ is a morphism from $(L_1,\alpha_1)$ to $(L_2,\alpha_2)$ if and only if $\alpha_2\circ f=f\circ\alpha_1$, that is the following diagram,}

\vspace{5mm}
\begin{center}
\begin{tikzpicture}[scale=2]
\draw(0,0) node[right] {$L_1$};
\draw[->] (.3,0)--(1,0);
\draw(1,0)node[right]{$L_2$};
\draw[->](1.2,-.14)--(1.2,-.8);
\draw(.54,.15)node[right]{$f$};
\draw(1.2,-.44)node[right]{$\alpha_2$};
\draw[->](.14,-.14)--(.14,-.8);
\draw(.14,-.44)node[left]{$\alpha_1$};
\draw(0,-.98)node[right]{$L_1$};
\draw[->](.3,-.98)--(1,-.98);
\draw(.54,-1.15)node[right]{$f$};
\draw(1.04,-.98)node[right]{$L_2$};

\end{tikzpicture}
\end{center}
\emph{commutes.}
\end{defn}
\begin{defn}
\emph{Furthermore will denote the category of regular Hom-Lie algebras as $\mathcal{RHLA}$.}
\end{defn}
\begin{rem}
\emph{From Lemma \ref{ho} and Lemma \ref{hoho} we can build a category isomorphism $\mathcal{L}:\mathcal{TL}\rightarrow\mathcal{RHLA}$.}
\end{rem}
\begin{defn}
\emph{We will also denote the category of all regular Hom-LTS as $\mathcal{RLTS}$. }
\end{defn}
\begin{defn}
\emph{From Lemma \ref{funclem} we can create a functor $\mathfrak{T}: \mathcal{TL}\rightarrow\mathcal{RLTS} $}
\end{defn}
\begin{rem}
\emph{Note the condition on the morphisms in the category $\mathcal{TL}$, is the same condition on the homomorphisms of a Hom-LTS (resp. Hom-Lie algebra), and thus $f$ is a morphism of the category $\mathcal{TL}$ if and only if $\mathfrak{T}f$ ($\mathfrak{L}f$) is a morphism of the category $\mathcal{RLTS}$ ($\mathcal{RHLA}$).}
\end{rem}

\subsection{Imbedding and Universality}
\begin{defn}
\emph{If $(T,[$ , , ],$\alpha_T)$ is a regular Hom-LTS and $(L,\alpha_L)$ is an object of $\mathcal{TL}$ then a Hom-LTS homomorphism $\epsilon: T\rightarrow\mathfrak{T}L$ is called an} imbedding \emph{of the regular Hom-LTS $(T,$[ , , ], $\alpha_T$) into $(L,\alpha_L)$.}
\end{defn}

\begin{defn}
\emph{ An imbedding $\nu:T\rightarrow \mathfrak{T}U$ into a Lie algebra $U$ is called \emph{universal} if for every imbedding $\epsilon:T\rightarrow\mathfrak{T}L$ there is a unique morphism of the objects $(L,\alpha_L)$ and $(U,\alpha_U)$, $\varphi:U\rightarrow L$ such that the diagram below is commutes. }
\end{defn}

\vspace{5mm}
\begin{center}
\begin{tikzpicture}[scale=2]
\draw(0,0) node[right] {T};
\draw[->] (.3,0)--(1,0);
\draw(1,0)node[right]{$\mathfrak{T}U$};
\draw[->](1.2,-.14)--(1.2,-.8);
\draw(.54,.15)node[right]{$\nu$};
\draw(1.2,-.44)node[right]{$\mathfrak{T}\varphi$};
\draw[->](.3,-.1)--(1,-.8);
\draw(.93,-.98)node[right]{$\mathfrak{T}L$};
\draw(.44,-.52)node[right]{$\epsilon$};
\end{tikzpicture}
\end{center}

So far no one has attempted a theory for this (to the authors knowledge). So we will be making some very reasonable assumptions, ones which align with the researchers in the area of interest, and developing a similar theory to that found in \cite{jacob} and \cite{oleg} for Lie triple systems.  This theory will be equivalent to the functor $\mathfrak{T}:\mathcal{TL}\rightarrow\mathcal{RLTS}$ having a left adjoint.

\subsection{Imbedding a Hom-LTS into a Lie Algebra}
Before moving on to our universal structure we will extend Lemma \ref{pooper} to build an imbedding of a regular Hom-LTS into a Lie Algebra. 

\begin{defn}\label{the}
\emph{Let $(T,[$ , , ],$\alpha$) be a regular Hom-LTS, then define $\mathcal{GHE}(T)\doteq\mathcal{IHD}(T)\oplus T$, with a product defined for $X+a,Y+b\in\mathcal{GHE}(T)$ as}
\begin{equation}\label{tooth}
[X+a,Y+b]\doteq[X,Y]+D_{ab}+X^{\alinv}(b)-Y^{\alinv}(b)
\end{equation}
\end{defn}
\begin{lem}\label{poop}
\emph{With the above definition $\mathcal{GHE}(T)$ is a Lie algebra, furthermore $\hat{\alpha}:\mathcal{GHE}(T)\rightarrow\mathcal{GHE}(T)$, defined as $\hat{\alpha}|_T\doteq\alpha$ and $\hat{\alpha}(D_{ab})=D_{\alpha(a)\alpha(b)}$ is a Lie algebra automorphism.}
\end{lem}

\begin{proof}
The verification that $\mathcal{GHE}(T)$ is a Lie algebra is a straightforward calculation and is left to the reader. Furthermore, we have
\begin{align*}
\hat{\alpha}([D_{ab}+c,D_{de}+f]) = & \hat{\alpha}(D_{d\alinv([a,b,e])}  + D_{\alinv([a,b,d])e}) \\
 & +\hat{\alpha}(D_{cf}+\alinv([a,b,f])-\alinv([d,e,c])) \\
= & D_{\alpha(d)[a,b,e]}+D_{[a,b,d]\alpha(e)}+D_{\alpha(c)\alpha(f)}+[a,b,f]-[d,e,c] 
\end{align*}
and
\begin{align*}
[\hat {\alpha}  (D_{ab}+c),\hat{\alpha}(D_{de}+f)]= & [D_{\alpha(a)\alpha(b)}+\alpha(c),D_{\alpha(d)\alpha(e)}+\alpha(f)] \\
= & D_{\alpha(d)[a,b,e]}+D_{[a,b,d]\alpha(e)}\\
&\;\;\; +D_{\alpha(c)\alpha(f)} + [a,b,f]-[d,e,c]
\end{align*}
for all $a,b,c,d,e,f\in T$, thus $\hat{\alpha}$ is a Lie algebra homomorphism. To see that it is an automorphism notice we may build a similar extension of $\alinv$ and that extension is indeed the inverse of $\hat{\alpha}$. 
\end{proof}

\begin{lem}
\emph{Let $($$T$, $[,,]$,$\alpha$$)$ be a regular Hom-LTS, then the canonical inclusion $\iota:T\rightarrow \mathcal{IHD}(T)\oplus T$ is an imbedding of $($$T$, $[,,]$,$\alpha$$)$ into ($\mathcal{GHE}(T),\hat{\alpha})$ and is called the} general imbedding \emph{of the regular Hom-LTS.}
\end{lem}
\begin{proof}
First note that $\iota\circ\alpha=\hat{\alpha}\circ\iota$, by the definition of $\hat{\alpha}$ in Lemma \ref{poop}. Now to see that $\iota([a,b,c])=\hat{\alpha}([[\iota(a),\iota(b)],\iota(c)])$, for all $a,b,c\in T$ so we calculate
\begin{equation*}
\hat{\alpha}([[\iota(a),\iota(b)],\iota(c)])= \hat{\alpha}([D_{ab},c])=\hat{\alpha}(\alinv([a,b,c]))=[a,b,c]=\iota([a,b,c]).
\end{equation*}
\end{proof}

\subsection{$\mathbb{Z}_2$-gradings}

Recall that a Lie algebra $L$ is said to be $\mathbb{Z}_2$-\emph{graded} if $L$ is a direct sum of a pair of $k$-submodules $L_0$ and $L_1$ such that $[L_i,L_j]\subset L_{i+j}$ for any $i,j\in\mathbb{Z}_2=\{0,1\}$. The decomposition $L=L_0\oplus L_1$ is called a $\mathbb{Z}_2$-\emph{grading} of $L$. If $L=L_0\oplus L_1$ and $K=K_0\oplus K_1$ are $\mathbb{Z}_2$-graded, a homomorphism $\varphi:L\rightarrow K$ is said to be \emph{graded} provided that $\varphi(L_i)\subset K_i$ for any $i\in\mathbb{Z}_2$.

It follows immediately from Definition \ref{the} and Lemma \ref{poop} that the direct sum decomposition of the algebra $\mathcal{GHE}(T)=\mathcal{IHD}(T)\oplus\ T$ is a $\mathbb{Z}_2$-grading and $\hat{\alpha}$ is a graded automorphism.

\section{The Universal Imbedding}
In this section it is our intention to construct a universal imbedding. The model for our treatment is the universal imbedding of a Lie triple system (see \cite{oleg}). 

\subsection{Construction of $\mathcal{U}(T)$}
Throughout this section $(T,[,,],\alpha)$ will be a  regular Hom-LTS. Our first objective is to build a central extension of $\mathcal{IHD}(T)$. 

Since $T$ is a $\mathcal{HDR}(T)$-module, we can consider the exterior product $T\wedge T$ as a $\mathcal{HDR}(T)$-module under the the unique action, motivated by (\ref{tooth}).
\begin{equation}\label{need2}
D\cdot (a\wedge b)=D^{\alinv}a\wedge b+a\wedge D^{\alinv}b
\end{equation}
for $D\in\mathcal{HDR}(T)$ and $a,b\in T$. Furthermore, formula (\ref{need1}) and identity (\ref{need2}) imply that the map $\lambda:T\wedge T\rightarrow \mathcal{HDR}(T)$ defined by $\lambda(a\wedge b)=D_{ab}$ is a module homomorphism from $T\wedge T$ to the adjoint module $\mathcal{HDR}(T)$ and that Im$(\lambda)=\mathcal{IHD}(T)$. 

It is easy to see that  $A(T\wedge T)\doteq \text{span}\{ \lambda(x)\cdot x \; : x\in T\wedge T\}$ is spanned by the elements of the form,
\begin{equation}\label{e1}
D_{ab}(a\wedge b)
\end{equation}
\begin{equation}\label{e2}
D_{ab}(c\wedge d)+D_{cd}(c\wedge d)
\end{equation}
for $a,b,c,d\in T$. From Lemma 3.1 in \cite{oleg} we have that $A(T\wedge T)$ is a sub module of $T\wedge T$. 

Next we define $\langle T,T\rangle$ as the quotient module $(T\wedge T)/ A(T\wedge T)$ and let $\langle a,b\rangle$ denote the coset containing $a\wedge b$. We then gain from Corollary 3.2 in \cite{oleg} that,
\begin{equation}
[\langle a, b\rangle, \langle c,d\rangle]=\langle D^{\alinv}_{ab}c, d\rangle + \langle c,  D^{\alinv}_{ab}d\rangle.
\end{equation} 
As well the map $\mu:\langle T, T\rangle\rightarrow\mathcal{IHD}(T)$ , defined by $\mu(\langle a, b\rangle)=D_{ab}$, is a central extension.

The following Lemma  is verified by a straightforward calculation and is thus omitted. 
\begin{lem}
\emph{The space $\mathcal{U}(T)\doteq\langle T,T\rangle\oplus T$ with the product}
\begin{equation}\label{boo}
[X+a,Y+b]=([X,Y]+\langle a,b\rangle)+(\mu^{\alinv}(X)b-\mu^{\alinv}(Y)a),
\end{equation}
\emph{where $X,Y\in\langle T,T\rangle$ and $a,b\in T$, is a $\mathbb{Z}_2$-graded Lie algebra. Moreover, the map $\nu:\mathcal{U}(T)\rightarrow\mathcal{IHD}(T)$, defined by $\nu(X+a)=\mu(X)+a$, is a central extension.}
\end{lem}
\begin{rem}
\emph{For what follows we will need a canonical realization of $\mathcal{U}(T)$ inside the category $\mathcal{TL}$. That is we will need to define an $\alpha_{\mathcal{U}}:\mathcal{U}(T)\rightarrow\mathcal{U}(T)$ which is an automorphism. The next Lemma defines this automorphism.}
\end{rem}
\begin{lem}\label{themoneyshot}
 \emph{Let $\alpha_{\mathcal{U}}:\mathcal{U}(T)\rightarrow\mathcal{U}(T)$, be an extension of $\alpha$, i.e. $\alpha_{\mathcal{U}}|_T\doteq \alpha$, and  we define its extension as $\alpha_{\mathcal{U}}(\langle a,b\rangle)\doteq\langle\alpha(a),\alpha(b)\rangle$ for all $a,b\in T$, then $\alpha_{\mathcal{U}}$ is an automorphism.}
\end{lem}
\begin{proof}
First we need to verify that $\alpha_{\mathcal{U}}$ is well defined on $\langle T,T\rangle$. It will therefore suffice to show that the map $\gamma:T\wedge T\rightarrow T\wedge T$ defined by $\gamma(a\wedge b)\doteq\alpha(a)\wedge\alpha(b)$ for all $a,b\in T$ is invariant with respect to $A(T\wedge T)$, that is $\gamma(X)\in A(T\wedge T)$, for all $X\in A(T\wedge T)$. We will verify this by checking that the elements which span $A(T\wedge T)$, that is (\ref{e1}) and (\ref{e2}) are invariant. So we may calculate for all $a,b,c,d\in T$,
\begin{align*}
\gamma(D_{ab}(a\wedge b))= & \gamma(\alinv([a,b,a])\wedge b)+\gamma(a\wedge\alinv([a,b,b])) \\
= & [a,b,a]\wedge\alpha(b) +\alpha(a)\wedge [a,b,b]\\
= & D_{\alpha(a)\alpha(b)}(\alpha(a)\wedge\alpha(b))\in A(T\wedge T)
\end{align*}
and
\begin{align*}
\gamma & (D_{ab}(c\wedge d)+D_{cd}(a\wedge b))\\
= & \gamma(\alinv([a,b,c])\wedge d)+\gamma(c\wedge\alinv([a,b,d]))\\
& + \gamma(\alinv([c,d,a])\wedge b)+\gamma(a\wedge\alinv([c,d,b]))\\
= & [a,b,c]\wedge\alpha(d)+\alpha(c)\wedge [a,b,d] \\
& + [c,d,a]\wedge\alpha(b)+\alpha(a)\wedge [c,d,b] \\
= & D_{\alpha(a)\alpha(b)}(\alpha(c)\wedge\alpha(d))+D_{\alpha(c)\alpha(d)}(\alpha(a)\wedge\alpha(b))\in A(T\wedge T).
\end{align*}
Next we check that $\alpha_{\mathcal{U}}$ is a Lie algebra homomorphism, 
\begin{align*}
\alpha_{\mathcal{U}}([\langle a,b\rangle+c,\langle d,e\rangle+f])= & \alpha_{\mathcal{U}}(\langle D_{ab}^{\alinv}d,e\rangle +\langle d,D_{ab}^{\alinv}e\rangle) \\
& + \alpha_{\mathcal{U}}(\langle c,f\rangle+D_{ab}^{\alinv} f-D_{de}^{\alinv} c) \\
= & \langle D_{ab}d,\alpha(e)\rangle+\langle \alpha(d),D_{ab}e\rangle \\
& + \langle \alpha(c),\alpha(f)\rangle +D_{ab}f-D_{de}c
\end{align*}
and
\begin{align*}
[\alpha_{\mathcal{U}}  (\langle a,b\rangle  + & c ),\alpha_{\mathcal{U}}(\langle d,e\rangle+f)]\\
 = & [\langle\alpha(a),\alpha(b)\rangle+\alpha(c),\langle\alpha(d),\alpha(e)\rangle+\alpha(f)]\\
 =& \langle D^{\alinv}_{\alpha(a)\alpha(b)}\alpha(d),\alpha(e)\rangle+\langle \alpha(d),D_{\alpha(a)\alpha(b)}^{\alinv}\alpha(e)\rangle \\
 & + \langle\alpha(c),\alpha(f)\rangle+D^{\alinv}_{\alpha(a)\alpha(b)}\alpha(f)-D^{\alinv}_{\alpha(d)\alpha(e)}\alpha(c)\\
 = & \langle D_{ab}d,\alpha(e)\rangle+\langle\alpha(d),D_{ab}e\rangle \\
 & + \langle\alpha(c),\alpha(f)\rangle+D_{ab}f-D_{de}c
\end{align*}
thus $\alpha_{\mathcal{U}}$ is a Lie algebra homomorphism.

Finally we have that  $\alpha_{\mathcal{U}}$ is an automorphism since we can build a similar extension of $\alinv$, and that extension is indeed the inverse of $\alpha_{\mathcal{U}}$.
\end{proof}

\subsection{Universal Property of $\mathcal{U}(T)$}
Next we consider $T$ as a subspace of $\mathcal{U}(T)$. It follows from (\ref{boo}) that $T$ generates $\mathcal{U}(T)$. 
\begin{thm}\label{main}
\emph{For every regular Hom-Lie triple system $(T,[$ , , ],$\alpha$) the canonical inclusion map $\iota_T:T\rightarrow \mathcal{U}(T)$ is a universal imbedding.}
\end{thm}

\begin{proof}
First it is straightforward consequence of Lemma \ref{themoneyshot} that $\iota_T\circ\alpha=\alpha_{\mathcal{U}}\circ\iota_T$, and we have
\begin{align*}
\alpha_\mathcal{U}([[\iota_T(a),\iota_T(b)],\iota_t(c)])= & \alpha_\mathcal{U}(D^{\alinv}_{ab}(c))= \alpha_{\mathcal{U}}(\alinv([a,b,c]))\\
= & [a,b,c]=\iota([a,b,c])
\end{align*}
for $a,b,c\in T$, thus $\iota_T$ is indeed an imbedding.

Now let $(L,$[ , ]) be a Lie algebra and let $\epsilon: T\rightarrow\mathfrak{T}L$ be an imbedding, i.e. there exists a Lie algebra automorphism $\alpha_L:L\rightarrow L$ such that for every $a,b,c\in T$,
\begin{equation}\label{yep}
\epsilon([a,b,c])=\alpha_L([[\epsilon(a),\epsilon(b)]_L,\epsilon(c)]_L)
\end{equation}
and
\begin{equation*}
\epsilon\circ\alpha= \alpha_L\circ\epsilon
\end{equation*}
and since both $\alpha$ and $\alpha_L$ are invertible (since they are automorphisms) we have
\begin{equation}\label{duh}
\epsilon\circ\alinv=\alinv_L\circ\epsilon.
\end{equation}
Thus by combining (\ref{yep}) and (\ref{duh}) we arrive at,
\begin{align*}
\epsilon(\alinv([a,b,c]))= & \epsilon([\alinv(a) ,\alinv(b),\alinv(c)])\\
= & \alpha_L([[\epsilon(\alinv(a)),\epsilon(\alinv(b))],\epsilon(\alinv(c))]_L)\\
=& \alpha_L([[\alinv_L(\epsilon(a)), \alinv_L(\epsilon(b))],\alinv_L(\epsilon(c))]_L)\\
= &(\alpha_L\circ\alinv_L)([[\epsilon(a),\epsilon(b)],\epsilon(c)]_L)\\
= & [[\epsilon(a),\epsilon(b)],\epsilon(c)]_L
\end{align*}
To prove the universality of the canonical inclusion we will need to construct a Lie algebra homomorphism $\varphi: \mathcal{U}\rightarrow L$ extending $\epsilon$ and proving that it is unique, then show that this is indeed a morphism in the category $\mathcal{TL}$, that is $\varphi\circ\alpha_{\mathcal{U}}=\alpha_L\circ\varphi$. 

To build this we first claim that there is a well-defined map $\langle\epsilon,\epsilon\rangle:\langle T,T\rangle\rightarrow L$ such that 
\begin{equation}\label{yeppers}
\langle\epsilon,\epsilon\rangle(\langle a, b\rangle)=[\epsilon(a),\epsilon(b)]_L
\end{equation} 
for every $a,b\in T$. To show that this map does not depend on the choice of representative (and is thus well defined) it will suffice to show that $A(T\wedge T)$ is in the kernel of the map $\zeta:T\wedge T\rightarrow L$ defined by $\zeta(x\wedge y)=[\epsilon(a),\epsilon(b)]_L$ for all $a,b\in T$. We will verify this by checking that the elements which span $A(T\wedge T)$, that is (\ref{e1}) and (\ref{e2}) are sent to zero by $\zeta$. First note that, 
\begin{equation*}
D_{ab}(a\wedge b)=\alinv([a,b,a])\wedge b+a\wedge\alinv([a,b,b])
\end{equation*}
and
\begin{align*}
D_{ab}(c\wedge d)+D_{cd}(a\wedge b)=  & \alinv( [a,b,c])\wedge d+c\wedge\alinv([a,b,d]) \\
& + \alinv([c,d,a])\wedge b+a\wedge\alinv([c,d,b])
\end{align*}
for all $a,b,c,d\in T$, and therefore we may calculate,
\begin{align*}
\zeta( \alinv & ([  a,b,a])\wedge b+a\wedge\alinv([a,b, b])) \\= & [\epsilon(\alinv([a,b,a])),\epsilon(b)]_L+[\epsilon(a),\epsilon(\alinv([a,b,b]))]_L\\
= & [[[\epsilon(a),\epsilon(b)]_L,\epsilon(a)]_L,\epsilon(b)]_L+[\epsilon(a)[[\epsilon(a),\epsilon(b)]_L,\epsilon(b)]_L]_L \\
= & [[\epsilon(a),\epsilon(b)]_L,[\epsilon(a),\epsilon(b)]_L]_L= 0 
\end{align*}
and
\begin{align*}
\zeta(\alinv &([a ,b,c])\wedge d+c\wedge\alinv([a,b,d])+\alinv([c,d,a])\wedge b+a\wedge\alinv([c,d,b])) \\
= & [\epsilon (\alinv([a,b,c]), \epsilon(d)]_L + [\epsilon(c),\epsilon(\alinv([a,b,d]))]_L \\
 & + [\epsilon(\alinv([c,d,a])),\epsilon(b)]_L + [\epsilon(a),\epsilon(\alinv([c,d,b]))]_L \\
 = & [[[\epsilon(a),\epsilon(b)]_L,\epsilon(c)]_L,\epsilon(d)]_L + [\epsilon(c),[[\epsilon(a),\epsilon(b)]_L,\epsilon(d)]_L]_L \\
 & + [[[\epsilon(c),\epsilon(d)]_L,\epsilon(a)]_L,\epsilon(b)]_L+[\epsilon(a),[[\epsilon(c),\epsilon(d)]_L,\epsilon(b)]_L]_L \\
 = & [[\epsilon(a),\epsilon(b)]_L,[\epsilon(c),\epsilon(d)]_L]_L+[[\epsilon(c),\epsilon(d)]_L,[\epsilon(a),\epsilon(b)]_L]_L =0.
\end{align*}

Now since $T$ generates the Lie algebra $\mathcal{U}(T)$, the identities (\ref{yep}) and (\ref{yeppers}) imply that the map $\varphi$, defined by 
\begin{equation}\label{var}
\varphi(X+a)\doteq\langle\epsilon,\epsilon\rangle(X)+\epsilon(a)
\end{equation}
for $X\in\langle T,T\rangle$ and $a\in T$, is a Lie algebra homomorphism. The next condition we need to show for $\varphi$ is that $\mathfrak{T}\varphi$ is a Hom-LTS homomorphism, that is $\varphi$ is a morphism of the category $\mathcal{T}\mathcal{L}$. So we need to show that
\begin{equation}\label{hush}
 \varphi\circ\alpha_{\mathcal{U}}=\alpha_L\circ\varphi. 
\end{equation} 
Yet for $a\in T$, $\epsilon(\alpha(a))=\alpha_L(\epsilon(a))$, thus (\ref{hush}) follows from Lemma \ref{themoneyshot} and the fact that $T$ generates $\mathcal{U}(T)$. Finally by construction we have $\varphi|_T=\epsilon$, and since $T$ generates $\mathcal{U}(T)$, there is only one such homomorphism $\varphi$.
\end{proof}

\section{A category of Lie algebras\\ equivalent to $\mathcal{RLTS}$}
In this section we determine a sub-category of $\mathcal{TL}$ equivalent to $\mathcal{RLTS}$. 

\subsection{The functor $\mathfrak{A}$}
The universal property of the algebra $\mathcal{U}(T)$ gives rise to the existence of a left adjoint to the forgetful functor $\mathfrak{T}:\mathcal{TL}\rightarrow\mathcal{RLTS}$ discussed in Section 3. That adjoint is the functor $\mathfrak{A}:\mathcal{RLTS}\rightarrow\mathcal{TL}$ which sends every regular Hom-LTS, $(T$, [ , , ],$\alpha_T$) to the algebra $\mathcal{U}(T)$ and every Hom-LTS homomorphism (to a Hom-LTS $(S$, [ , , ],$\alpha_S$)) $\vartheta:T\rightarrow S$  to the morphism $\varphi:\mathcal{U}(T)\rightarrow\mathcal{U}(S)$ defined by (\ref{var}) for the imbedding $\epsilon=\iota_S\circ\vartheta$. These formulas also imply that the functor $\mathfrak{A}$ is faithful. These statements do need a bit of verification. For instance we need to show that $\epsilon\circ\alpha_T=\alpha_{\mathcal{U}(S)}\circ\epsilon$. Yet this is easily verified since $\vartheta$ is a Hom-LTS homomorphism, that is $\vartheta\circ\alpha_T=\alpha_S\circ\vartheta$ and thus we have,
\begin{equation*}
\epsilon\circ\alpha_T=\iota_S\circ\vartheta\circ\alpha_T=\iota_S\circ\alpha_S\circ\vartheta=\alpha_{\mathcal{U}(S)}\circ\iota_S\circ\vartheta=\alpha_{\mathcal{U}(S)}\circ\epsilon.
\end{equation*}

It turns out the most natural setting is to consider $\mathfrak{A}$ as a functor to the category $\mathcal{TL}_{\mathbb{Z}_2}$ which has objects as pairs $(L, \alpha_L)$ where $L$ is a  $\mathbb{Z}_2$ graded Lie algebra and $\alpha_L$ is a graded automorphism, and with morphisms as graded homomorphisms with the commutative relation from section 3. This is most natural since this makes $\mathfrak{A}$ both full and faithful.  Since $\mathcal{TL}_{\mathbb{Z}_2}$ is a subcategory of $\mathcal{TL}$ we will refrain from changing notations, and denote the functor $\mathfrak{A}: \mathcal{RLTS}\rightarrow\mathcal{TL}_{\mathbb{Z}_2}$. Notice this makes sense, since $\alpha_{\mathcal{U}}$ is indeed a graded automorphism.


Since $\mathfrak{A}:\mathcal{RLTS}\rightarrow\mathcal{TL}_{\mathbb{Z}_2}$ is full and faithful, the category $\mathcal{RLTS}$ is equivalent to the full subcategory of $\mathcal{TL}_{\mathbb{Z}_2}$ whose objects are the Lie algebras isomorphic to the algebras of the form $\mathcal{U}(T)$ paired with appropriate graded Lie automorphisms. This is analogous to the results in \cite{oleg}, thus we may use Corollary 4.8 in \cite{oleg} and obtain our final result.

\begin{thm}\label{duh}
\emph{The category $\mathcal{RLTS}$ is equivalent to the category of pairs $(L,\alpha)$ where $L$ is a centrally 0-closed $\mathbb{Z}_2$-graded Lie algebra generated by its odd component, and $\alpha$ is a graded automorphism.}
\end{thm}
\begin{rem}
\emph{In the above theorem since $L$ is generated by its odd component, $\alpha$ is uniquely defined by its action on $L_0$, this can be noticed since for any $l\in L_1$ there exists $l_i\,l'_i\in L_0$ such that}
\begin{equation*}
l=\sum[l_i,l'_i]
\end{equation*}
\emph{thus}
\begin{equation*}
\alpha(l)=\sum[\alpha(l_i),\alpha(l'_i)].
\end{equation*}
\end{rem}

%

\section*{Acknowledgment}
I am indebted to Oleg N. Smirnov whom suggested this project and has been a constant mentor and advisor. Also I would like to give my gratitude to Deanna M. Caveny-Noecker for helping me with revisions and being there to listen to me think these topics out loud.

\end{document}